\pdfoutput=1

\documentclass[12pt, a4paper]{article}
\usepackage[margin=1.5in, top=1.5in, bottom=1.5in]{geometry}
\usepackage{graphicx}
\usepackage{titlesec}
\usepackage{amsmath}
\usepackage{amssymb}
\usepackage{amsthm}
\usepackage{cases}
\usepackage[english]{babel}
\usepackage{hyperref}
\usepackage{enumitem}
\usepackage{authblk}
\usepackage{dsfont}
\usepackage{sectsty}
\usepackage{xcolor}

\sectionfont{\centering}

\theoremstyle{break}
	\newtheorem{thm}{Theorem}[section]
\theoremstyle{break}
	\newtheorem{asm}{Assumption}[section]
\theoremstyle{break}
	\newtheorem{rem}{Remark}[section]
\theoremstyle{break}
	\newtheorem{lem}{Lemma}[section]
\theoremstyle{break}
	\newtheorem{defi}{Definition}[section]

\makeatletter
\newcommand{\mylabel}[2]{#2\def\@currentlabel{#2}\label{#1}}
\makeatother

\newcommand{\Leo}{\mathcal{L}_{\varepsilon}}
\newcommand{\Seq}[2]{\left\{#1_#2\right\}_{#2 = 0}^\infty}

\newcommand{\Leb}{L^2\left(\Omega\right)}

\newcommand{\CDR}{C^1\left(\mathbb{R}\right)}

\counterwithin{equation}{section}

\title{Long time behaviour of solutions to non-local and non-linear dispersion problems}

\author{Maciej Tadej \\ \href{mailto:maciej.tadej@math.uni.wroc.pl}{maciej.tadej@math.uni.wroc.pl}}

\affil{\textit{Insitute of Mathematics, University of Wrocław, pl. Grunwaldzki 2, 50-384 Wrocław}}

\date{}

\begin{document}

\maketitle

\hspace{20pt}

\begin{abstract}
This paper explores a non-linear, non-local model describing the evolution of a single species. We investigate scenarios where the spatial domain is either an arbitrary bounded and open subset of the $n$-dimensional Euclidean space or a periodic environment modeled by $n$-dimensional torus. The analysis includes the study of spectrum of the linear, bounded operator in the considered equation, which is a scaled, non-local analogue of classical Laplacian with Neumann boundaries. In particular we show the explicit formulas for eigenvalues and eigenfunctions. Moreover we show the asymptotic behaviour of eigenvalues. Within the context of the non-linear evolution problem, we establish the existence of an invariant region, give a criterion for convergence to the mean mass, and construct spatially heterogeneous steady states.
\end{abstract}


\newpage

\section{Introduction}

\subsection{Statement and main results}

In this paper we study properties of solutions $u := u(x,t)$ of the following non-linear, non-local model describing the evolution of a density of single species.
\begin{numcases}{}
u_t(x,t) = \Leo u(x,t) + f(u(x,t)) & $(x,t) \in \Omega \times (0, \infty),$ \label{disersion_evolution} \\
u(x,0) = u_0(x) & $x \in \Omega,$ \label{disersion_initial_condition}
\end{numcases}
where $\Omega \subset \mathbb{R}^n$ is a bounded and open domain. We assume that an initial datum $u_0 \in \Leb $, a scaling parameter $\varepsilon > 0$ and the force term $f \in \CDR$. The dispersion operator $\Leo: \Leb \rightarrow \Leb$ is given by the formula
\begin{align}
\Leo u(x,t) = \int_\Omega J_{\varepsilon}(x-y)\left(u(y,t) - u(x,t)\right)\:dy,\label{dispersion_operator}
\end{align} 
where $J_\varepsilon : \mathbb{R}^n \rightarrow \mathbb{R}_{\geq 0}$ is a properly scaled version of the integral kernel $J: \mathbb{R}^n \rightarrow \mathbb{R}_{\geq 0}$. Here, we distinguishh two frameworks stated in the following assumptions.
\begin{asm}[General case]
    Denote by $\Omega \subset \mathbb{R}^n$ an open and bounded region. We consider the integral kernels $J \in C(\mathbb{R}^n)$ satisfying the following properties
    \begin{enumerate}[label=\textnormal{(\arabic*)}]
        \item [\mylabel{J1}{(J1)}] $J(z) \geq 0$ for $z \in \mathbb{R}^n$ and $J(0) > 0$, 
        \item [\mylabel{J2}{(J2)}] $J(z) = J(-z)$ for $z \in \mathbb{R}^n$,
        \item [\mylabel{J3}{(J3)}] $J(z) \rightarrow 0$ as $|z| \rightarrow \infty$,
        \item [\mylabel{J4}{(J4)}] $\int_{\mathbb{R}^n} J(z)|z|^2dz < \infty$.
        \item [\mylabel{J5}{(J5)}] $\int_{\mathbb{R}^n} J(z) dz = 1$.
    \end{enumerate}
    We introduce $\varepsilon >0$ and $m \in [0, 2]$ to define the re-scaled kernel $J_\varepsilon \in C(\mathbb{R}^n)$
    \begin{equation}
        J_\varepsilon(z) = \frac{C_{norm}}{\varepsilon^{m+n}}J\left(\frac{z}{\varepsilon}\right), \quad C_{norm} = \left( \frac{1}{2} \int_{\mathbb{R}^n} J(z) |z|^2 \:dz \right)^{-1}.
    \end{equation} \label{assumptions_general_kernel}
\end{asm}
Now we define the setting where dispersion outside the domain is identified by the reflection of boundaries.
\begin{asm}[Periodic case]
    Denote by $\mathbb{T}^n = (-\pi, \pi)^n$ an $n$-dimensional torus. We consider the integral kernels $J \in C(\mathbb{R}^n)$ satisfying claims \ref{J1} - \ref{J3} of Assumption \ref{assumptions_general_kernel}. Next we introduce $\varepsilon > 0$ and $m \in [0,2]$ to define the re-scaled, truncated and periodized kernel $J_\varepsilon \in C(\mathbb{T}^n)$ given by
    \begin{equation}
        J_\varepsilon(z) = \frac{C_{norm}^\varepsilon}{\varepsilon^{m+n}}J_{per}\left(\frac{z}{\varepsilon} \right), \quad C_{norm}^\varepsilon = \left(\frac{1}{2} \int_{\left( -\frac{\pi}{\varepsilon}, \frac{\pi}{\varepsilon} \right)^n} J_{per}(z)|z|^2\:dz\right)^{-1}.
    \end{equation}
    We extended the kernel periodically in such a way that $J_{per}\left(\frac{z}{\varepsilon}\right)$ is a $2\pi$-periodic function and so $J_{per}(z)$ is $\frac{2\pi}{\varepsilon}$-periodic kernel.
     \label{assumptions_periodic_kernel}
\end{asm}

The study of non-local dispersion, initiated by Hutson et al. \cite{Hutson2003, Hutson2005} and Fife \cite{Fife2001} in ecological contexts, characterizes population density $u := u(x,t)$ in the spatial domain $\Omega \subset \mathbb{R}^n$ over infinite time. The dispersion operator $\Leo$ captures population influx and efflux at location $x$. Assumption \ref{assumptions_general_kernel} outlines confined dispersal without boundary flux, posing technical challenges. In \cite{Hutson2003} a simpler, periodic framework was introduced as in Assumption \ref{assumptions_periodic_kernel}. This considers movements outside the boundary as reflections back into the domain $\Omega$. Further comments on related results follow after a brief overview of the main results.

\hspace{1pt}

We underscore the significance of analyzing the spectrum $\sigma(\Leo)$ of the dispersion operator $\Leo$ in this work. In addition to addressing the continuous part of this set, we aim to solve the corresponding eigenvalue problem
\begin{equation}
    \Leo \varphi_k = \beta_k \varphi_k.
    \label{eigenvalue_problem}
\end{equation}
In Theorem \ref{thm:dispersion_spectrum} below, we characterize $\sigma(\Leo)$ when $\Omega$ and $J_\varepsilon$ satisfy Assumption \ref{assumptions_general_kernel}. In the periodic setting defined by Assumption \ref{assumptions_periodic_kernel}, Theorem \ref{thm:dispersion_eigenpairs} provides explicit formulas for $(\beta_k, \varphi_k)$. Additionally, we explore the asymptotic behavior of the obtained eigenvalues as $|k| \rightarrow \infty$ and $\varepsilon \rightarrow 0$.

\hspace{1pt}

The spatially homogeneous asymptotic behavior of solutions to the Cauchy problem \eqref{disersion_evolution}-\eqref{disersion_initial_condition} is the focus of our next investigation. We seek an invariant region $\Gamma \subset \mathbb{R}$. Specifically, we demonstrate that every solution $u$ satisfies the following relation almost everywhere in $\Omega$ for all $t > 0.0$.
\begin{equation}
    u_0(x) \in \Gamma \implies u(x, t) \in \Gamma.
    \label{invariant_region_definiton}
\end{equation}
The existence of an invariant region $\Gamma \subset \mathbb{R}$ is confirmed in Theorem \ref{thm:invariant_region} when the force term $f$ has two distinct zeros. Let $\beta_1$ denote the largest non-zero eigenvalue of $\Leo$. We assume that the following condition holds true:
\begin{equation}
    \sigma = 2\left(\beta_1 + \max_{u \in \Gamma} |f'(u)|\right) < 0.
\end{equation}
Given this, we conclude in Theorem \ref{thm:mean_mass_convergence} the convergence to the average mass 
\begin{equation}
    u(x, t) \xrightarrow{t \rightarrow \infty} \frac{1}{|\Omega|} \int_\Omega u(x, t) \:dx \quad \text{ in $L^2(\Omega)$}.
    \label{convergence_to_mean_mass}
\end{equation}

\hspace{1pt}

Finally, we investigate the existence and stability of spatially heterogeneous equilibria. In particular we solve the following steady state problem 
\begin{align}
    \Leo U + f(U) = 0.
    \label{dispersion_equilibria}
\end{align}
The continuous steady states are constructed in Theorem \ref{thm:continous_ss_existence} through bifurcation techniques. Additionally, in Theorem \ref{thm:discontinous_ss}, we establish the existence of discontinuous yet stable solutions, under the assumption that the force term $f$ has two distinct zeros $u_1 < u_2$ with $f'(u_1) < 0$ and $f'(u_2) < 0$.

\subsection{Related results}

The local counterpart of the Cauchy problem \eqref{disersion_evolution}-\eqref{disersion_initial_condition} is given by the following reaction diffusion problem with no flux through boundaries

\begin{numcases}{}
u_t(x,t) = \Delta u(x,t) + f(u(x,t)) & $(x,t) \in \Omega \times (0, \infty),$ \label{diffusion} \\
\frac{\partial u}{\partial n} = 0 & $x \in \partial\Omega$, \\
u(x,0) = u_0(x) & $x \in \Omega,$ \label{diffusion_initial_condition}
\end{numcases}

where in this contexts, $n$ represents the outward normal vector to $\Omega$. This model is already well known. We refer readers interested in general properties of such problems to \cite{Evans2010, Wei-Ming2011, Quittner-Souplet2007}. In \cite{Conway1978} the invariant regions were discussed with their applications to establishing spatially homogeneous asymptotics.  We emphasize the classical works \cite{Matano1979, CastenHolland1978, Chaffee1975} where the existence and instability of non-constant steady states is explored. 

The link between models \eqref{disersion_evolution}-\eqref{disersion_initial_condition} and \eqref{diffusion}-\eqref{diffusion_initial_condition} was established when $f \equiv 0$. Namely the convergence of solutions $u^\varepsilon$ from the non-local problem to solutions $u$ of the local problem was demonstrated in \cite{Cortazar2008}. We also direct the reader to a more general setting, considering the convergence of non-local dispersion problems to the local parabolic counterparts, discussed in \cite{MolinoRossi2016}.

We underscore the well-posed nature of problem \eqref{disersion_evolution}-\eqref{disersion_initial_condition} in the Hadamard sense. In the linear case, existence and uniqueness were established in \cite{Cortazar2008}, while the non-linear case was addressed in \cite{Sun2022}.

The eigenvalue problem for the operator \eqref{dispersion_operator}, was studied for instance in \cite{GarciaMelianRossi2009}. Authors demonstrated that $\beta_1$ is strictly negative. However, the question of the existence of a corresponding eigenfunction remains unresolved. Additional insights into this problem are presented in \cite{Foss2018, Alfaro2023}. More general approach, outlined in \cite{Coville2010, Coville2016, BerestyckiCovilleVo2016}, yields only the trivial eigenpair when applied to our framework.

In the linear non-local case, the work of \cite{ChasseigneChavesRossi2006} explores spatially homogeneous asymptotic behavior of solutions $u^\varepsilon$ of problem \eqref{disersion_evolution}-\eqref{disersion_initial_condition}. The non-local steady problem \eqref{dispersion_equilibria} was investigated in \cite{ChasseigneChavesRossi2006} and \cite{Coville2010}. 

Other related works encompass diverse applications to real world, including image processing \cite{GilboaOsher2007, GilboaOsher2009}, epidemiology \cite{Yang2019}, ecology \cite{Hadeler1999}, and various mathematical articles \cite{Biler2019, SlevinskyMontanelli2018, Sun2023, Sun2022}.

\section{Spectral analysis} \label{sec:spectral_analysis}

\subsection{Main results of the section}

Here we perform the spectral analysis of the operator $\Leo: L^2(\Omega) \rightarrow L^2(\Omega)$ given by formula \eqref{dispersion_operator}. For that we decompose the dispersion operator

\begin{equation}
    \Leo u(x) = \int_\Omega J_\varepsilon(x-y) u(y)\:dy + b_\varepsilon(x)u(x),
    \label{leo_decomposition}
\end{equation}
where
\begin{equation}
    b_\varepsilon(x) = -\int_\Omega J_\varepsilon(x-y)\:dy. \label{b_epsilon}
\end{equation}

We characterise the spectrum $\sigma(\Leo)$ on an arbitrary bounded domain $\Omega$.

\begin{thm}[General case]
    Let $J_\varepsilon$ and $\Omega$ satisfy Assumption \ref{assumptions_general_kernel}. Denote $\Leo: \Leb \rightarrow \Leb$ as in a formula \eqref{leo_decomposition}. Then the spectrum of $\Leo$ is given by
    \begin{equation}
        \sigma(\Leo) = b_\varepsilon \left[\: \overline{\Omega} \:\right] \cup \Seq{\beta}{k},
        \label{EQ:spectrum_formula_general}
    \end{equation}
    where
    $$b_\varepsilon\left[\: \overline{\Omega} \:\right] = \left\{ - \int_\Omega J_\varepsilon(x-y) \:dy \:\Big|\: x \in \overline{\Omega }\right\}$$
    and the sequence $\Seq{\beta}{k}$ is a set of isolated eigenvalues of $\Leo$. \label{thm:dispersion_spectrum}
\end{thm}

Next we consider the spatial domain $\Omega = \mathbb{T}^n$  being an $n$-dimensional torus. In this approach we conclude the following

\begin{rem}
    Let $J_\varepsilon$ and $\mathbb{T}^n$ satisfy Assumption \ref{assumptions_periodic_kernel}. Denote the set $b_\varepsilon \left[\: \overline{\Omega} \:\right]$ as in Theorem \ref{thm:dispersion_spectrum}. Then this set is a single point. In particular we have
    $$b_\varepsilon \left[\: \overline{\Omega} \:\right] = \left\{ - \hat{J}_\varepsilon(0) \right\},$$
    where $\hat{J}_\varepsilon(k) = \int_{\mathbb{T}^n} J_\varepsilon(z) e^{ikz}\:dz$ is a $k$-th Fourier coefficient of the kernel $J_\varepsilon$. \label{rem:dispersion_periodic_spectrum}
\end{rem}
In this setting we will also derive the explicit formulas for the eigenvalues $\Seq{\beta}{k}$ and their corresponding eigenfunctions $\Seq{\varphi}{k}$ from the Theorem \ref{thm:dispersion_spectrum}.

\begin{thm}[Periodic case]
Let $J_\varepsilon$ and $\mathbb{T}^n$ satisfy Assumption \ref{assumptions_periodic_kernel}. Denote $\Leo: L^2\left(\mathbb{T}^n\right) \rightarrow L^2\left(\mathbb{T}^n\right)$ as in a formula \eqref{leo_decomposition}. Then the spectrum of $\Leo$ is given by
\begin{equation}
    \sigma(\Leo) = \left\{ -\hat{J}_\varepsilon(0) \right\} \cup \left\{\beta\right\}_{k=0}^\infty.
    \label{EQ:spectrum_formula_periodic}
\end{equation}
Eigenvalues and the corresponding eigenfunctions are given explicitly as follows
\begin{equation}
\beta_k(\varepsilon) := \hat{J}_\varepsilon(k) - \hat{J}_\varepsilon(0) = \int_{\mathbb{T}^n} J_\varepsilon(z) (\cos(k\cdot z) - 1)\:dz,
\label{EQ:eigenvalues_formula_periodic}
\end{equation}
\begin{equation}
\varphi_k^1(x) = \cos(k\cdot x), \quad \varphi_k^2(x) = \sin(k\cdot x),
\label{EQ:eigenfunctions_formula_periodic}
\end{equation}
where $\hat{J}_\varepsilon(k) = \int_{\mathbb{T}^n} J_\varepsilon(z) e^{ikz}\:dz$ is a $k$-th Fourier coefficient of the kernel $J_\varepsilon$. Moreover the following properties hold
\begin{enumerate}[label=\textnormal{(\arabic*)}]
    \item [\mylabel{thm:2.2.1:i}{(i)}] The eigenvalues $\{\beta_k\}_{k=0}^\infty$ are real, bounded and non-positive for $k \neq 0$.
    \item [\mylabel{thm:2.2.1:ii}{(ii)}] The limit eigenvalue is an element of the spectrum. In particular
    \begin{equation}
        \beta_k \xrightarrow{|k| \rightarrow \infty} \beta_\infty(\varepsilon) = - \hat{J}_\varepsilon(0) = b_\varepsilon(0).
        \label{EQ:limit_eigenvalue_b_infty_definition}
    \end{equation}
    When $\varepsilon \rightarrow 0$ then the limit eigenvalue is bounded in a following way
    \begin{equation}
        -\varepsilon^{-m} \leq \beta_\infty(\varepsilon) \leq -\frac{\varepsilon^{2-m}}{\pi^2}.
        \label{EQ:limit_eigenvalue_b_infty_bounds}
    \end{equation}
    \item [\mylabel{thm:2.2.1:iii}{(iii)}] When $\varepsilon \rightarrow 0$ and $m = 2$ then the eigenvalues $\{\beta_k\}_{k=0}^\infty$ are given explicitly by
    \begin{equation}
        \lim_{\varepsilon \rightarrow 0} \beta_k(\varepsilon) = -\frac{|k|^2}{n}
    \end{equation}
    If $m \in [0, 2)$ then the limit becomes zero
\end{enumerate} 
\label{thm:dispersion_eigenpairs}
\end{thm}
The following subsections contain the proofs of theorems \ref{thm:dispersion_spectrum} and \ref{thm:dispersion_eigenpairs}.

\subsection{Proof of Theorem \ref{thm:dispersion_spectrum}}

Firstly we show that the set $b_\varepsilon \left[\: \overline{\Omega} \:\right]$ belongs to the spectrum of $\Leo$.

\begin{lem} \label{lem:spectrum_inclusion}
    Let $J_\varepsilon$ and $\Omega$ satisfy Assumption \ref{assumptions_general_kernel}. Then for $b_\varepsilon$ defined as in formula \eqref{b_epsilon} we obtain the inclusion
    \begin{equation*}
        b_\varepsilon\left[\:\overline{\Omega}\:\right] \subseteq \sigma(\Leo).
    \end{equation*}
\end{lem} 
\begin{proof}
    Let $\lambda \in b_\varepsilon\left[\:\Omega\:\right]$. Suppose that $\lambda \notin \sigma(\Leo)$. Then for all $v \in \Leb$ there is a solution $u \in \Leb$ to the equation
    $$\Leo u (x) - \lambda u(x) = v(x).$$
    As a consequence of bounded inverse theorem, there exists $C > 0$ such that
    \begin{equation}
        C \|u\|_2 \leq \|\Leo u - \lambda u\|_2 \label{coercivity_eq}.
    \end{equation}
    By the continuity of $J$, there exists $x_\lambda \in \Omega$ such that $b_\varepsilon(x_\lambda) = \lambda$. Hence we take $u_R^\lambda(x) = R^{-n/2} \mathds{1}_{B(x_\lambda, R)}$ and easily verify that it is a function of unit norm and it satisfies the upper bound of a following form
    \begin{align*}
        \|\Leo u_R^{\lambda} - \lambda u_R^{\lambda}\|_2
        &\leq \left\| \left( \int_{\Omega} J_{\varepsilon}(\cdot-y) u_R^{\lambda}(y)\:dy \right)\right\|_2 + \|\left(b_\varepsilon - \lambda\right) u_R^\lambda\|_2 \\
        &= \left(\int_\Omega\left(\int_{B(x_\lambda, R)} J_\varepsilon(x-y) R^{-n/2} \:dy\right)^2\:dx \right)^{1/2} \\
        &+ \left(\int_{B(x_\lambda, R)}\left(b_\varepsilon(x) - \lambda \right)^2dx\right)^{1/2} \\
        &\leq \|J_{\varepsilon}\|_{\infty} |\Omega|^{1/2} R^{n/2} + \left(\int_{B(x_\lambda, R)}\left(b_\varepsilon(x) - \lambda \right)^2 \:dx\right)^{1/2} \xrightarrow{R \rightarrow 0} 0.
    \end{align*}
    But it contradicts \eqref{coercivity_eq}, which yields $\lambda \in \sigma(\Leo)$. We conclude saying that $\sigma(\Leo)$ must be a closed set since it is a spectrum of bounded linear operator.
\end{proof}
Next, we show that every other element of the spectrum is an eigenvalue of $\Leo$.
\begin{lem} \label{lem:spectrum_remaining_elements}
    Let $J_\varepsilon$ and $\Omega$ satisfy Assumption \ref{assumptions_general_kernel}, $\varepsilon > 0$, $m \in \mathbb{N}$. Then every $\lambda \in \sigma(\Leo) \setminus b_\varepsilon\left[\:\overline{\Omega}\:\right]$ is an eigenvalue of the operator $\Leo$ given in \eqref{dispersion_operator}.
\end{lem}
\begin{proof}
    Denote a set $D = \mathbb{C} \setminus b_\varepsilon \left[\: \overline{\Omega} \:\right]$. This set is open and connected. Now we re-write equation \eqref{eigenvalue_problem} using formula \eqref{leo_decomposition} to obtain
    \begin{equation*}
        \int_\Omega J_\varepsilon(x-y) u(y)\:dy = \left( \beta - b_\varepsilon(x)\right) u(x).
    \end{equation*}
    We divide both sides of the equation to get the fixed point problem and denote the analytic, operator valued function $F: D \rightarrow \mathcal{L}(L^2(\Omega))$ as follows
    \begin{equation*}
        F(\beta)[u](x) := \int_\Omega \frac{J_\varepsilon(x-y)}{\beta - b_\varepsilon(x)}u(y)\:dy = u(x).
    \end{equation*}
    Now, let us verify the convergence of the difference quotient
    \begin{align*}
        \frac{F(\beta) - F(\beta_0)}{\beta - \beta_0} [u] (x) 
        &= \frac{1}{\beta - \beta_0}\left(\int_\Omega \frac{J_\varepsilon(x-y)}{\beta - b_\varepsilon(x)}u(y)\:dy - \int_\Omega \frac{J_\varepsilon(x-y)}{\beta_0 - b_\varepsilon(x)}u(y)\:dy\right) \\
        &= -\int_\Omega \frac{J_\varepsilon(x-y)}{(\beta - b_\varepsilon(x))(\beta_0 - b_\varepsilon(x))}u(y)\:dy \\
        &\xrightarrow{\beta_0 \rightarrow \beta} -\int_\Omega \frac{J_\varepsilon(x-y)}{(\beta - b_\varepsilon(x))^2}u(y)\:dy,
    \end{align*}
    which proves that $F: D \rightarrow \mathcal{L}(\Leb)$ is an analytic operator, as claimed.
    
    We emphasize that since $\beta$ never intersects the range of $b_\varepsilon$ then the integral kernel $K(x,y) = \frac{J_\varepsilon(x-y)}{\beta - b_\varepsilon(x)}$ has a finite $L^2(\Omega \times \Omega)$ norm and so it is a Hilbert-Schmidt operator. In particular it is a compact operator. 
    
    Now, we verify that $F$ is a contraction for some $\beta \in D$. For instance we can take $\beta_0 \in D$ such that  
    $$\Re(\beta_0) = - \inf_{x \in \Omega} - b_\varepsilon(x) + 2 \|J_\varepsilon\|_\infty |\Omega|^\frac{1}{2}, \quad \Im(\beta_0) = 0.$$ 
    Having that we can bound the operator norm of $F(\beta_0)$ to obtain the following
    \begin{align*}
        \| F(\beta_0) \| &= \sup_{\|u\|_2=1} \|F(\beta_0)[u]\|_2 \leq \sup_{\|u\|_2 = 1} \left( \int_\Omega \int_\Omega \frac{\|J_\varepsilon\|^2 u^2(y)}{(\beta_0 - b_\varepsilon(x))^2}\:dy\:dx \right)^\frac{1}{2} \leq \frac{1}{2}
    \end{align*}
    Hence $(I - F(\beta_0))^{-1}$ exists. But now, by the Analytical Fredholm Theorem \cite{ReedSimon}, we know that $(I - F(\beta))^{-1}$ exists for all $\beta \in D \setminus \Seq{\beta}{k}$. On the other hand, for any element $\beta_k \in \Seq{\beta}{k}$ there exists a corresponding function $\varphi_k \in \Leb$ such that $(\beta_k, \varphi_k)$ solves the equation \eqref{eigenvalue_problem}. Thus $\beta_k$ is an eigenvalue.
\end{proof}
Combining the results obtained in the Lemma \ref{lem:spectrum_inclusion} and \ref{lem:spectrum_remaining_elements} we finish the proof of Theorem \ref{thm:dispersion_spectrum}.

\subsection{Proof of Theorem \ref{thm:dispersion_eigenpairs}}

Firstly we observe that by Theorem \ref{thm:dispersion_spectrum} and Remark \ref{rem:dispersion_periodic_spectrum} the spectrum is given by formula \eqref{EQ:spectrum_formula_periodic} as claimed. To find the latter part of the spectrum we denote the following complex valued function defined for $x \in \mathbb{T}^n, k \in \mathbb{Z}^n$

\begin{equation*}
    \psi_k(x) = e^{ik\cdot x}.
\end{equation*}

We plug $\psi_k$ into the equation \eqref{eigenvalue_problem} and use the periodicity of $J_\varepsilon$ to obtain

\begin{align*}
    \Leo \psi_k(x) &= \int_{\mathbb{T}^n} J_\varepsilon(x-y) e^{ik\cdot y}\:dy - \int_{\mathbb{T}^n} J_\varepsilon(x-y)\:dy\: \psi_k(x) \\
    &= \left(\int_{\mathbb{T}^n} J_\varepsilon(x-y) e^{ik\cdot (y-x)}\:dy - \int_{\mathbb{T}^n} J_\varepsilon(x-y)\:dy\right) \psi_k(x) \\
    &= \left(\int_{\mathbb{T}^n + x} J_\varepsilon(z) e^{ik\cdot z}\:dz - \int_{\mathbb{T}^n + x} J_\varepsilon(z)\:dz\right) \psi_k(x) \\
    &= \left( \int_{\mathbb{T}^n} J_\varepsilon(z) e^{ik\cdot z}\:dz - \int_{\mathbb{T}^n} J_\varepsilon(z)\:dz\right) \psi_k(x) \\
    &= \left(\hat{J}_\varepsilon(k) - \hat{J}_\varepsilon(0)\right) \psi_k(x).
\end{align*}
Using the symmetry of $J$, we verify formula \eqref{EQ:eigenvalues_formula_periodic}. Indeed we have that
\begin{equation*}
    \beta_k(\varepsilon) := \hat{J}_\varepsilon(k) - \hat{J}_\varepsilon(0) = \int_{\mathbb{T}^n} J_\varepsilon(z) (\cos(k\cdot z) - 1)\:dz,
\end{equation*}
as asserted in the Theorem. We conclude that both real and imaginary part of $\psi_k$ solves the equation \eqref{eigenvalue_problem} and so
$$\varphi_k^1(x) := \Re(\psi_k) = \cos(k \cdot x),$$
$$\varphi_k^2(x) := \Im(\psi_k) = \sin(k \cdot x),$$
which is exactly formula \eqref{EQ:eigenfunctions_formula_periodic} as we claimed.

\textit{(i)} It follows immediately from formula \eqref{EQ:eigenvalues_formula_periodic}.

\textit{(ii)} The Riemann-Lebesgue Lemma applied to formula \eqref{EQ:eigenvalues_formula_periodic} yields
\begin{equation*}
    \beta_\infty(\varepsilon) := \lim_{|k| \rightarrow \infty} \beta_k(\varepsilon) = -\hat{J}_\varepsilon(0) = b_\varepsilon(0).
\end{equation*}
Next we compute
\begin{align*}
    \beta_\infty(\varepsilon) &= 
    - \int_{\mathbb{T}^n} J_\varepsilon(x)\:dx \\
    &= -\frac{C_{norm}^\varepsilon}{\varepsilon^{m+n}} \int_{\left(-\pi, \pi\right)^n} J_{per}\left(\frac{x}{\varepsilon}\right)\:dx \\
    &= -\frac{C_{norm}^\varepsilon}{\varepsilon^{m}} \int_{\left(-\frac{\pi}{\varepsilon}, \frac{\pi}{\varepsilon}\right)^n} J_{per}\left(z\right)\:dz \\
    &= -\frac{1}{\varepsilon^m} \frac{\int_{\left(-\frac{\pi}{\varepsilon}, \frac{\pi}{\varepsilon}\right)^n} J_{per}\left(z\right)\:dz}{\int_{\left(-\frac{\pi}{\varepsilon}, \frac{\pi}{\varepsilon}\right)^n} J_{per}\left(z \right)|z|^2\:dz} =: -\frac{\rho(\varepsilon)}{\varepsilon^m},
\end{align*}
where we denoted
\begin{equation*}
    \rho(\varepsilon) = \frac{\int_{\left(-\frac{\pi}{\varepsilon}, \frac{\pi}{\varepsilon}\right)^n} J_{per}\left(z\right)\:dz}{\int_{\left(-\frac{\pi}{\varepsilon}, \frac{\pi}{\varepsilon}\right)^n} J_{per}\left(z \right)|z|^2\:dz}.
\end{equation*}
We bound this quantity from below in the following way
\begin{align*}
    \rho(\varepsilon) &\geq \frac{\int_{\left(-\frac{\pi}{\varepsilon}, \frac{\pi}{\varepsilon}\right)^n} J_{per}\left(z\right)\:dz}{\frac{\pi^2}{\varepsilon^2}\int_{\left(-\frac{\pi}{\varepsilon}, \frac{\pi}{\varepsilon}\right)^n} J_{per}(z)\:dz} = \frac{\varepsilon^2}{\pi^2}.
\end{align*}
Reversely we see that for $\varepsilon > 0$ sufficiently small we have

\begin{align*}
    \rho(\varepsilon) = \frac{
    \int_{|z| < 1} J_{per}\left(z\right)\:dz + 
    \int_{\substack{\left(-\frac{\pi}{\varepsilon}, \frac{\pi}{\varepsilon}\right)^n \\ |z| > 1}} J_{per}\left(z\right)\:dz
    }{
        \int_{|z| < 1} J_{per} \left(z\right)|z|^2\:dz + 
    \int_{\substack{\left(-\frac{\pi}{\varepsilon}, \frac{\pi}{\varepsilon}\right)^n \\ |z| > 1}} J_{per}\left(z\right)|z|^2\:dz
    } \leq 1.
\end{align*}

Hence we obtained the bound for a limit eigenvalue
\begin{equation*}
    -\varepsilon^{-m} \leq \beta_\infty(\varepsilon) \leq -\frac{\varepsilon^{2-m}}{\pi^2}.
\end{equation*}

\textit{(iii)} We use the Taylor series expansion to re-write $\beta_k$ with $k \neq 0$.
\begin{align*}
    \beta_k(\varepsilon)
    &= \int_{(-\pi, \pi)^n} J_\varepsilon(x)\cos(k\cdot x)\:dx - \int_{(-\pi, \pi)^n} J_\varepsilon(x)\:dx \\
    &= - \int_{(-\pi, \pi)^n} J_\varepsilon(x) \frac{(k\cdot x)^2}{2}\: dx + \sum_{l=2}^\infty (-1)^l \int_{(-\pi, \pi)^n} J_\varepsilon(x)\frac{(k\cdot x)^{2l}}{(2l)!}\:dx
\end{align*}

Substitution $z = x / \varepsilon$ gives a simplified formula in the asymptotic form.
\begin{align*}
    \beta_k(\varepsilon)   &= -\frac{\varepsilon^{2-m} C_{norm}^\varepsilon}{2} \int_{\left(-\frac{\pi}{\varepsilon}, \frac{\pi}{\varepsilon}\right)^n} J_{per}(z) (k \cdot z)^2 \:dz + o(\varepsilon^{4-m}) \\
    &= -\frac{\varepsilon^{2-m} C_{norm}^\varepsilon}{2} \sum_{i=1}^n \int_{\left(-\frac{\pi}{\varepsilon}, \frac{\pi}{\varepsilon}\right)^n} J_{per}(z) k_i^2z_i^2 \:dz\\
    &-\frac{\varepsilon^{2-m} C_{norm}^\varepsilon}{2} \sum_{i\neq j}^n \int_{\left(-\frac{\pi}{\varepsilon}, \frac{\pi}{\varepsilon}\right)^n} J_{per}(z) k_ik_jz_iz_j \:dz + o(\varepsilon^{4-m})\\
    &= - \varepsilon^{2-m} \sum_{i=1}^n k_{i}^2 \frac{\int_{\left(-\frac{\pi}{\varepsilon}, \frac{\pi}{\varepsilon}\right)^n} J_{per} z_i^2\:dz}{\int_{\left(-\frac{\pi}{\varepsilon}, \frac{\pi}{\varepsilon}\right)^n} J_{per} |z|^2\:dz} + o(\varepsilon^{4-m}) \\
    &=  - \varepsilon^{2-m}\frac{|k|^2}{n}  + o(\varepsilon^{4-m})
\end{align*}
This ends the proof of Theorem \ref{thm:dispersion_eigenpairs}.

\section{Existence, invariant regions and spatially homogeneous asymptotics} \label{sec:existence_invariant_regions_homo_asymptotics}

\subsection{Main results of the section}

We begin by stating results on asymptotic properties of solutions to the initial value problem \eqref{disersion_evolution}-\eqref{disersion_initial_condition}.

\begin{thm}[Existence of solutions]
    Consider the integral kernel $J_\varepsilon$ and domain $\Omega$ satisfying either Assumptions \ref{assumptions_general_kernel} or \ref{assumptions_periodic_kernel}. Assume that $f \in C^1(\mathbb{R})$ such that $f' \in L^\infty(\Omega)$. The following statements are true
    \begin{enumerate}[label=\textnormal{(\arabic*)}]
        \item [\mylabel{thm:eu_i}{(i)}] For every $u_0 \in \Leb$ there exists a unique solution $u \in C\left([0, \infty), \Leb\right)$ to the Cauchy problem \eqref{disersion_evolution}-\eqref{disersion_initial_condition}.
        \item [\mylabel{thm:eu_ii}{(ii)}] Solutions $u, v$ to the Cauchy problem \eqref{disersion_evolution}-\eqref{disersion_initial_condition} depend continuously on the initial data $u_0, v_0$. In particular for all $T > 0$ there exists $C(T) > 0$ such that
        \begin{equation}
            \|u - v\|_2 \leq C(T)\|u_0 - v_0\|_2.
        \end{equation}
        \item [\mylabel{thm:eu_iii}{(iii)}] If $u_0 \in C(\overline{\Omega})$ then both $u \in C\left(\overline{\Omega} \times [0, T]\right)$ and $u_t \in C\left(\overline{\Omega} \times [0, T]\right)$ for each $T > 0$
    \end{enumerate}
    \label{thm:existence_uniqueness}
\end{thm}

Next we study existence of invariant regions in the considered evolution problem.

\begin{thm}[Invariant region]
    Let the assumptions of the Theorem \ref{thm:existence_uniqueness} hold true. Additionally assume there exists $u_1 < u_2$ with $f(u_1) = f(u_2) = 0$. Let $u \in C([0, \infty), \Leb)$ be a solution of initial value problem \eqref{disersion_evolution}-\eqref{disersion_initial_condition} with an initial datum $u_0 \in \Leb$. Denote the set $\Gamma = [u_1, u_2]$. If $u_0(x) \in \Gamma$ a.e. then $u(x, t) \in \Gamma$ a.e. for all $t > 0$.
    \label{thm:invariant_region}
\end{thm}

And finally we derive a criterion for the spatially homogeneous asymptotics.

\begin{thm}[Convergence to the mean]
    Let the assumptions of the Theorem \ref{thm:invariant_region} hold true.
    Let $u \in C\left([0, \infty), \Leb\right)$ be a solution of the Cauchy problem \eqref{disersion_evolution} - \eqref{disersion_initial_condition} with an initial datum $u_0 \in L^2\left(\Omega\right)$. Denote the mean mass
    \begin{equation*}
        u_\Omega(t) = \frac{1}{|\Omega|}\int_\Omega u(x,t)\:dx.
    \end{equation*}There are positive numbers $a_1 > 0, a_2 >0$ and a critical value given by formula
    \begin{equation}
        \sigma = 2(\beta_1 + \max_{u \in \Gamma}|f'(u)|),
        \label{critical_value}
    \end{equation}
    where $\beta_1$ is largest non-trivial eigenvalue. For these values the following statements hold true for all $t > 0$.
    \begin{enumerate}[label=\textnormal{(\arabic*)}]
        \item [\mylabel{thm:conv_i}{(i)}] The difference between the solution $u$ and a mean mass $u_\Omega$ is estimated by
        \begin{equation}
            \| u(t) - u_{\Omega}(t) \|_{2} \leq a_1 e^{\sigma t}.
        \end{equation}
        \item [\mylabel{thm:conv_ii}{(ii)}] The mean mass $u_\Omega$ satisfies the following ordinary differential equation
        \begin{equation}
            \begin{cases}
                u_\Omega'(t) &= f(u_\Omega(t)) + O(e^{\sigma t}), \\
                u_\Omega(0) &= \frac{1}{|\Omega|}\int_\Omega u_0(x)\: dx.
            \end{cases}
        \end{equation}
        \item [\mylabel{thm:conv_iii}{(iii)}] The Dirichlet form associated with $\Leo$ is estimated by
        \begin{equation}
            |\left\langle \Leo u(t), u(t) \right\rangle| \leq a_2 e^{2\sigma t}.
        \end{equation}
    \end{enumerate}
    \label{thm:convergence_to_mean_mass}
\end{thm}

\begin{rem}
    Observe that once the value of $\sigma$ given by formula \eqref{critical_value} is negative then Theorem \ref{thm:convergence_to_mean_mass} yields the convergence to the mean mass in the sense of $L^2$\ norm. What is more, we see by Theorem \ref{thm:dispersion_eigenpairs} that this criterion becomes
    \begin{equation*}
        \max_{u \in \Gamma} |f'(u)| < \varepsilon^{2-m} \frac{|k|^2}{n} + o(\varepsilon^{4-m}).
    \end{equation*}
\end{rem}

\begin{rem}
    The constants $a_1, a_2$ in Theorem \ref{thm:convergence_to_mean_mass} are given explicitly by
    \begin{equation*}
        a_1 = \|u_0 - u_{\Omega}(0)\|_2, \quad a_2 = \|\Leo\|_{L^2 \rightarrow L^2} \|u_0 - u_{\Omega}(0)\|_2^2. \label{thm:mean_mass_convergence}
    \end{equation*}
\end{rem}

Reminder of this section is focused on the proofs of the above theorems.

\subsection{Proof of Theorem \ref{thm:existence_uniqueness}}

Observe that \ref{thm:eu_i} is a consequence of a more general result by Sun \cite{Sun2022}, who constructed a unique solution $u \in C([0, \infty), L^1(\Omega))$. This solution is obtained as a fixed point of the following equation
\begin{equation}
    u(x,t) = u_0(x) + \int_0^t \Leo u(x,s) \:ds + \int_0^t f(u(x,s))\:ds.
    \label{ala_duhamel_formula}
\end{equation}
Obviously $L^2(\Omega) \subset L^1(\Omega)$ for a bounded domain $\Omega$ and hence \ref{thm:eu_i} follows. Denote by $w = u - v$ the difference of solutions $u, v$ to initial value problem \eqref{disersion_evolution}-\eqref{disersion_initial_condition} and proceed for $t \in [0, T]$ as follows
\begin{align*}
    \| w(t) \|_2 &\leq \|w_0\|_2 + \int_0^t \|\Leo w (s)\|_2 + \|f(u(s)) - f(v(s))\|_2 \:ds \\
    &\leq \|w_0\|_2 + \int_0^t\left(\int_\Omega \int_\Omega J_\varepsilon^2(x-y)\left(w(y, s) - w(x, s)\right)^2\:dy\:dx\right)^{1/2}\:ds \\
    &+ \int_0^t\left(\int_\Omega \left(f(u(x,s)) - f(v(x,s))\right)^2\:dx\right)^{1/2}\:ds \\
    &\leq \|w_0\|_2 + \left( \sqrt{2}\|J_\varepsilon\|_\infty |\Omega|^{1/2} + \|f'\|_\infty \right) \int_0^t \|w(s)\|_2 \:ds.
\end{align*}
The Gronwall inequality yields
\begin{equation*}
    \|w(t)\|_2 \leq e^{T \left( \sqrt{2}\|J_\varepsilon\|_\infty |\Omega|^{1/2} + \|f'\|_\infty \right)} \|w_0\|_2.
\end{equation*}

This proves \ref{thm:eu_ii} with $C(T) = e^{T \left( \sqrt{2}\|J_\varepsilon\|_\infty |\Omega|^{1/2} + \|f'\|_\infty \right)}$. Now, let $u_0 \in C(\overline{\Omega})$. Denote by $\mathbb{X} = C\left(\overline{\Omega} \times [0, T_0]\right)$ the Banach space equipped with the norm
\begin{equation*}
    ||| u ||| = \max_{t \in [0, T_0]} \max_{x \in \overline{\Omega}} |u(x,t)|.
\end{equation*}
Now apply the Banach Fixed Point Theorem and the usual procedure. We deduce that there exists unique solution $u \in C\left( \overline{\Omega} \times [0, T] \right)$. But now observe that second and third terms in formula \eqref{ala_duhamel_formula} become continuously differentiable in time, and so in fact $u_t \in C\left( \overline{\Omega} \times [0, \infty) \right)$ which is \ref{thm:eu_iii}. Now the proof of Theorem \ref{thm:existence_uniqueness} is finished.

\subsection{Proof of Theorem \ref{thm:invariant_region}}

Our strategy to show an existence of an invariant region is to differentiate maximum and minimum of the solution to \eqref{disersion_evolution}-\eqref{disersion_initial_condition}. The first step is to solve the problem when $u_0 \in C(\overline{\Omega})$. Indeed we have

\begin{lem}
    For $T > 0$ and bounded domain $\Omega \subset \mathbb{R}^n$ let $u \in C\left(\overline{\Omega}, [0, T]\right)$ such that $u_t \in C\left(\overline{\Omega}, [0, T]\right)$. Denote the functions $\xi: [0, T] \rightarrow \overline{\Omega}$, $\zeta: [0, T] \rightarrow \overline{\Omega}$ and $\mu: [0, T] \rightarrow \mathbb{R}$, $\nu: [0, T] \rightarrow \mathbb{R}$ given by the following formulas
    \begin{equation}
        \mu(t) := u(\xi(t), t) = \max_{x \in \Omega} u(x,t), \quad \nu(t) := u(\zeta(t), t) = \min_{x \in \Omega} u(x,t)
        \label{mu_nu_xi_zeta_formula}
    \end{equation}
    Then both $\mu$ and $\nu$ are differentiable a.e. and their derivatives are given by
    \begin{equation}
        \mu'(t) = u_t(\xi(t), t), \quad \nu'(t) = u_t(\zeta(t), t).
        \label{mu_nu_derivatives_formula}
    \end{equation} 
    \label{lem:mu_nu_differentiability}
\end{lem}

\begin{proof}
    Let $t > 0, s > 0$. Assume that $\mu(t) > \mu(s)$. formula \eqref{mu_nu_xi_zeta_formula} yields
    \begin{equation*}
        0 < \mu(t) - \mu(s) \leq u(\xi(t), t) - u(\xi(t), s) \leq \max_{(x,t) \in \overline{\Omega} \times [0, T]}|u_t(x,t)||t-s|
    \end{equation*}
    which yields Lipschitz continouity of $\mu$. By the Radamacher theorem \cite{Evans} we conclude that $\mu$ is differentiable almost everywhere. Now, we estimate the difference quotient as follows
    \begin{align*}
        \frac{\mu(t+h) - \mu(t)}{h} 
        &= \frac{u(\xi(t+h),t+h) - u(\xi(t), t)}{h} \\
        &\leq \frac{u(\xi(t+h),t+h) - u(\xi(t+h), t)}{h} \\
        &= u_t(\xi(t+h), c(t, h)),
    \end{align*}
    where by the Mean Value Theorem $c(t, h) \in [t-h, t+h]$. Passing to the limit as $h \rightarrow 0$ we get
    \begin{equation}
        \mu'(t) \leq u_t(\xi(t), t).
        \label{upper_bound_mu}
    \end{equation}
    The reverse estimate can be written down as follows
    \begin{align*}
        \frac{\mu(t+h) - \mu(t)}{h}
        &\geq \frac{u(\xi(t),t+h) - u(\xi(t), t)}{h} \\
        &= u_t(\xi(t), c(t, h)).
    \end{align*}
    Similarly as before, we pass to the limit as $h \rightarrow 0$ and obtain
    \begin{equation}
        \mu'(t) \geq u_t(\xi(t), t).
        \label{lower_bound_mu}
    \end{equation}
    Combining inequalities \eqref{upper_bound_mu} and \eqref{lower_bound_mu} we get formula \eqref{mu_nu_derivatives_formula} for $\mu'$. The proof for $\nu'$ is analogous. That completes the proof of the Lemma \ref{lem:mu_nu_differentiability}.
\end{proof}

Next we show that $\mu$ and $\nu$ are bounded by the minimum and maximum of $u_0$.

\begin{lem}
    Let the assumptions of the Theorem \ref{thm:invariant_region} hold true. Let $u \in C\left(\overline{\Omega}, [0, T]\right)$ be a solution to the Cauchy problem \eqref{disersion_evolution}-\eqref{disersion_initial_condition} with an initial datum $u_0 \in C(\overline{\Omega})$. Denote the functions $\xi, \zeta, \mu, \nu$ as in the Lemma \ref{lem:invariant_region_cont}. Then for all $t \geq 0$ the following estimates hold true
    \begin{equation*}
        u_1 \leq \nu(t) \leq \mu(t) \leq u_2
    \end{equation*}
    \label{lem:invariant_region_cont}
\end{lem}

\begin{proof}
    Recall the definition of the dispersion operator $\Leo$ as in formula \eqref{dispersion_operator}. In maximum point $\xi$ we have
    \begin{equation*}
        \Leo u(\xi(t), t) = \int_\Omega J_\varepsilon(x-y) \left(u(y, t) - u(\xi(t), t)\right)\:dy \leq 0.
    \end{equation*}
    Moreover in minimum point $\zeta$ we have
    \begin{equation*}
        \Leo u(\zeta(t), t) = \int_\Omega J_\varepsilon(x-y) \left(u(y, t) - u(\zeta(t), t)\right)\:dy \geq 0.
    \end{equation*}
    By Lemma \ref{lem:mu_nu_differentiability} both $\mu$ and $\nu$ are differentiable almost everywhere in time $t > 0$. Formula \eqref{mu_nu_derivatives_formula} yields
    $$\mu'(t) = u_t(\xi(t), t) = \Leo u(\xi(t), t) + f(\mu(t)) \leq f(\mu(t)).$$
    Similarly we have
    $$\nu'(t) = u_t(\zeta(t), t) = \Leo u(\zeta(t), t) + f(\nu(t)) \geq f(\nu(t)).$$
    The proof is completed by the usual comparison principle for the differential inequalities \cite{Petrovitch1901}.
\end{proof}
 
In order to proceed with the proof of Theorem \ref{thm:invariant_region} observe that every $u_0 \in L^2(\Omega)$ can be approximated by the sequence $u^n_0 \in C(\overline{\Omega})$. Now let the functions $u \in C([0, T], \Leb)$ and $u^n \in C(\overline{\Omega}, [0, T])$ be the solutions of the Cauchy problem \eqref{disersion_evolution}-\eqref{disersion_initial_condition} equipped respectively with the initial data $u_0, u^n_0$. Assume that $u_0(x) \in \Gamma$ a.e., then $u_0^n \in \Gamma$ a.e.. By the Theorem \ref{thm:existence_uniqueness} \ref{thm:eu_ii} we have the following estimate

\begin{equation*}
    \|u - u^n\|_2 \leq C(T)\|u_0 - u_0^n\|_2.
\end{equation*}

Passing to the limit as $n \rightarrow \infty$ we get that $u^n \rightarrow u$ in the sense of $L^2$ norm. But now one can take the subsequence $u^{n_k}$ such that $u^{n_k} \rightarrow u$ a.e. as $k \rightarrow \infty$. By the Lemma \ref{lem:invariant_region_cont} $u^{n_k}(x,t) \in \Gamma$ a.e. for all $t \in [0, T]$. Hence $u(x,t) \in \Gamma$ a.e. for all $t \in [0, T]$. Thus the proof of Theorem \ref{thm:invariant_region} is completed.

\subsection{Proof of Theorem \ref{thm:convergence_to_mean_mass}}

We introduce the notion of energy $E_v$ of the distance function $v$ by the formula

\begin{equation*}
    E_v(t) = \frac{1}{2}\int_\Omega |v(x,t)|^2 \:dx, \quad v(x,t) = u(x,t) - u_\Omega(t),
\end{equation*}

where $u$ is a solution to Cauchy problem \eqref{disersion_evolution}-\eqref{disersion_initial_condition} and $u_\Omega(t)$ is the mean mass.

We differentiate the energy to get

\begin{align*}
    E_v'(t) &= \int_\Omega v_tv \:dx \\
    &= \int_\Omega \left(u_t - u_\Omega'\right)v\:dx \\
    &= \int_\Omega \left(\Leo u + f(u) - \frac{1}{|\Omega|} \int_\Omega f(u)\:dy\right)v\:dx \\
    &= \left\langle \Leo u, v\right\rangle + \int_\Omega \left( f(u) - f(u_\Omega) - \frac{1}{|\Omega|} \int_\Omega \left(f(u) - f(u_\Omega)\right)\:dy\:\right) v \: dx \\
    &\leq \left\langle \Leo v, v\right\rangle + \max_{u \in \Gamma}|f'(u)| E_v(t) - \frac{1}{|\Omega|} \int_\Omega \left(f(u) - f(u_\Omega)\right)\:dy \int_\Omega v\:dx\\
    &= \langle \Leo v, v\rangle + \max_{u \in \Gamma}|f'(u)| E_v(t)
\end{align*}

Now we divide both sides by $E_v(t)$ and obtain the final inequality

\begin{align*}
    \frac{E_v'(t)}{E_v(t)} &\leq 2\left(\sup_{v \in L^2(\Omega), \int_\Omega v = 0}\frac{\langle \Leo v, v \rangle}{\|v\|_2^2} + \max_{u \in \Gamma}|f'(u)|\right) \\
    &= 2\left(\beta_1 + \max_{u \in \Gamma}|f'(u)|\right) =: \sigma.
\end{align*}

Solving this differential inequality we get that

\begin{equation*}
    \|u(t) - u_\Omega(t)\|_2 \leq \|u_0 - u_{\Omega}(0)\|_2 e^{\sigma t} =: a_1 e^{\sigma t}.
\end{equation*}

Next, we compute the derivative of the mean mass $u_\Omega$, apply the Mean Value Theorem and Cauchy-Schwarz inequality to obtain

\begin{align*}
    u_\Omega'(t) 
    &= \frac{1}{|\Omega|} \int_\Omega f(u)\:dx \\
    &\leq f(u_\Omega) + \frac{1}{|\Omega|} \int_\Omega \left|f(u) - f(u_\Omega)\right|\:dx \\
    &\leq f(u_\Omega) + \frac{\max_{u \in \Gamma} |f'(u)|}{|\Omega|} \int_\Omega |u - u_\Omega|\:dx \\
    &\leq f(u_\Omega) + \max_{u \in \Gamma} |f'(u)| \|v\|_2^2
\end{align*}

Finally, using claim \ref{thm:conv_i} we get that mean mass satisfies the following differential inequality

\begin{align*}
    \begin{cases}
        u_\Omega'(t) = f(u_\Omega(t)) + O(e^{\sigma t}) \\
        u_\Omega(0) = \frac{1}{|\Omega|}\int_\Omega u_0(x)\:dx
    \end{cases}    
\end{align*}

Lastly, it suffices to use \ref{thm:conv_i} once more to get

\begin{align*}
    \left| \langle \Leo u, u \rangle \right| &= \left| \langle \Leo v, u \rangle \right|\\
    &= \left| \langle \Leo v, v \rangle + u_\Omega(t) \int_\Omega \Leo v \:dx \right| \\
    &= \left|\langle \Leo v, v \rangle\right| \\
    &\leq \|\Leo v\|_2 \|v\|_2 \\
    &\leq \|\Leo \| \|v\|_2^2 \\
    &\leq \|\Leo\|u_0 - u_{0, \Omega}\|_2^2 e^{2\sigma t} = a_2 e^{2\sigma t},
\end{align*}

The proof of Theorem \ref{thm:mean_mass_convergence} is now completed.

\section{Spatially heterogeneous steady states} \label{sec:hetero_asymptotics}

\subsection{Main results of the section}

In this section we construct non-constant solution $U$ of steady equation \eqref{dispersion_equilibria}. The first class of continuous stationary states is given by the following theorem.

\begin{thm}[Continuous equilibrium]
    Let $J_\varepsilon$ and $\mathbb{T}^n$ satisfy Assumption \ref{assumptions_periodic_kernel}. Assume that $f \in C^1(\mathbb{R})$ is such that there exist $u_k \in \mathbb{R}$ satisfying $f(u_k) = 0$, $f'(u_k) = \beta_k$, where $\beta_k$ is a $k$-th eigenvalue of dispersion operator $\Leo$. Consider the equation
    \begin{equation}
        \Leo U + f(U) = \lambda U.
        \label{bifurcation_equation}
    \end{equation}
    Then for $\ell > \frac{n}{2}$ there exists a sequence of non-constant solutions $\Seq{U}{i} \subset W^{\ell,2}(\mathbb{T}^n)$ and corresponding numbers $\Seq{\lambda}{i} \subset \mathbb{R}$ to equation \eqref{bifurcation_equation} such that
    \begin{equation*}
        (U_i, \lambda_i) \xrightarrow{i \rightarrow \infty} (u_k, 0).
    \end{equation*}
    \label{thm:continous_ss_existence}
\end{thm}

Now we derive the class of discontinuous equilibria in the following theorem.

\begin{thm}[Discontinuous equilibrium]
    Consider the integral kernel $J_\varepsilon$ and domain $\Omega$ satisfying either Assumption \ref{assumptions_general_kernel} or \ref{assumptions_periodic_kernel}. Assume that $f \in C^2_b(\mathbb{R})$ is such that there exists $u_1, u_2$ satisfying $f(u_1) = f(u_2) = 0$ and $f'(u_1) \neq 0$, $f'(u_2) \neq 0$. For arbitrary open, non-empty $\Omega_1 \subset \Omega$ and $\Omega_2 \subset \Omega$ such that
    \begin{equation}
        \Omega_1 \cap \Omega_2 = \emptyset \quad \text{ and } \Omega = \left(\overline{\Omega_1 \cup \Omega_2}\right)^o. \label{omega_partition}
    \end{equation}
    Suppose that $\varepsilon > 0$ is sufficiently large. Then the steady problem \eqref{dispersion_equilibria} admits a solution of form $U = \hat{U} + \varphi$. The jump function $\hat{U}$ is defined by formula
    \begin{equation}
        \Tilde{U}(x) = \begin{cases}
            u_1 \quad x \in \Omega_1, \\
            u_2 \quad x \in \Omega_2,
        \end{cases}
        \label{jump_function}
    \end{equation}
    and the perturbation $\varphi \in \overline{B(0, R)} \subset \L^\infty(\Omega)$ for a small radius $R > 0$. If we additionally assume that $f'(u_1) < 0$ and $f'(u_2) < 0$ and $R$ is sufficiently small, then $U$ is linearly, exponentially stable in $L^2(\Omega)$.
    \label{thm:discontinous_ss}
\end{thm}

The remaining part of this section is dedicated to the proofs of both theorems.

\subsection{Proof of Theorem \ref{thm:continous_ss_existence}}

We start with an abstract equation with linear operator $L$ and nonlinear $H$
\begin{equation}
    L v + H(v) = \lambda v,
    \label{abstract_equation_bifurcation}
\end{equation}
where $H$ is such that $H(v) = o(\|v\|)$ when $\|v\| \rightarrow 0$. Obviously there is a trivial solution $(0, \lambda)$ for all $\lambda \in \mathbb{R}$. If equation \eqref{abstract_equation_bifurcation} has a variational formulation then bifurcation techniques allow us to construct a family of its non-trivial solutions.

\begin{defi}[Bifurcation point]
    We call $(0, \lambda)$ a bifurcation point for equation \eqref{abstract_equation_bifurcation} if every neighbourhood of $(0, \lambda)$ contains a non-zero solution.
    \label{def:bifurcation_point}
\end{defi}

Now, for convenience of the reader we supply this subsection with the classical result by Rabinowitz \cite{Rabinowitz1977}.

\begin{thm}[Rabinowitz Bifurcation Theorem]
    Let $\mathbb{H}$ be a Hilbert space, $V$ a neighbourhood of $0$ in $\mathbb{H}$, and $I \in C^2\left(V, \mathbb{R}\right)$ with $DI(u) = Lu + H(u)$ with $L: \mathbb{H} \rightarrow \mathbb{H}$ being linear and potential $H(u) = o(\|u\|)$ when $\|u\| \rightarrow 0$. If $\lambda$ is an isolated eigenvalue of $L$ of finite multiplicity, then $(0, \lambda)$ is a bifurcation point of the equation \eqref{abstract_equation_bifurcation}.
    Moreover, at least one of the following alternatives occur.
    \begin{itemize}
        \item [(i)] $(0, \lambda)$ is not an isolated solution in $\mathbb{H} \times \{\lambda\} $,
        \item [(ii)] There is a one-sided neighbourhood, $\Lambda$, of $\lambda$ such that for all $\mu \in \Lambda \setminus \{\lambda\}$ there are at least two distinct nontrivial solutions,
        \item [(iii)] There is a neighbourhood, $\Lambda$, of $\lambda$ such that for all $\mu \in \Lambda \setminus \{\lambda\}$ there is at least one nontrivial solution.
    \end{itemize}
    \label{thm:rabinowitz_bifurcation}
\end{thm}

We show commutation of $\Leo$ with weak derivative when environment is periodic.

\begin{lem}
    Let $\mathbb{T}^n, J_\varepsilon$ satisfy the Assumption \ref{assumptions_periodic_kernel}. Denoting by $D^{\alpha}$ the weak derivative with multi-index $\alpha \in \mathbb{Z}_{\geq 0}^n$ the following equation holds true
    \begin{equation}
        D^\alpha \Leo u = \Leo D^\alpha u.
    \end{equation}
    \label{lem:commutation_formula}
\end{lem}

We skip a direct proof of Lemma \ref{lem:commutation_formula}. Next, we proceed with properties of the non-linear functional $I$ defined by equation \eqref{nonlinear_functional}.

\begin{lem}
    Let $\mathbb{T}^n, J_\varepsilon$ satisfy either Assumption \ref{assumptions_general_kernel} or \ref{assumptions_periodic_kernel}. Assume that for $f \in C^1(\mathbb{R})$ there exists a primitive $F' = f$. Define a nonlinear functional $I:\Theta \rightarrow \mathbb{R}$, where $\Theta \subset W^{k, 2}(\mathbb{T}^n)$ and $k > \frac{n}{2}$ is an open neighbourhood of $0$, as follows
    \begin{align}
        \begin{split}
            I(u) 
            &= \frac{1}{2}\int_{\mathbb{T}^n \times \mathbb{T}^n} J_\varepsilon(x-y)\left(u(y) - u(x)\right)^2\:dy\:dx + \int_{\mathbb{T}^n} F(u(x)) \:dx. \label{nonlinear_functional}
        \end{split}
    \end{align}
    Then $I \in C^2(\Theta, \mathbb{R})$. Moreover, the Frechete derivative is given by formula
    \begin{equation*}
        DI(u) = \Leo u + f(u).
    \end{equation*}
    \label{lem:nonlinear_functional}
\end{lem}

We skip a standard proof of Lemma \ref{lem:nonlinear_functional}. Next we leverage the bifurcation theory to obtain the following result

\begin{lem}
    Let the assumptions of Theorem \ref{thm:continous_ss_existence} hold true. Without loss of generality suppose that $f(0) = 0$ and $f'(0) = -\beta_k$, where $\beta_k$ is a $k$-th eigenvalue of dispersion operator $\Leo$. Take $\ell \geq \frac{n}{2}$. Then every neighbourhood of $(0, 0) \in W^{\ell,2}(\mathbb{T}^n) \times \mathbb{R}$ possesses a non-trivial solution to equation \eqref{bifurcation_equation}.
    \label{lem:bifurcation_existence}
\end{lem}

\begin{proof}
    See that Lemma \ref{lem:commutation_formula} yields that $\Leo$ is a bounded, linear mapping of $W^{\ell,2}(\mathbb{T}^n)$ onto itself. Re-write equation \eqref{bifurcation_equation} using Taylor's Theorem to obtain
    \begin{equation*}
        \Leo u + f'(0) u + R_1(u) = 0,
    \end{equation*}
    and finally we consider the equation perturbed by the factor $\lambda u$ for some $\lambda \in \mathbb{R}$
    \begin{equation}
        \Leo u + R_1(u) = (\lambda + \beta_k)u.
        \label{transformed_bifurcation_equation}
    \end{equation}
    The Taylor's Theorem yields that $R_1(u) = o(\|u\|_\infty)$ as $\|u\|_\infty \rightarrow 0$. Thus the embedding $W^{\ell,2}(\mathbb{T}^n) \subset L^\infty(\mathbb{T}^n)$ yields that $R_1$ is a desired potential operator. By Lemma \ref{lem:nonlinear_functional} considered problem has the variational formulation. Since by Theorem \ref{thm:dispersion_spectrum} $\beta_k$ is an isolated eigenvalue of $\Leo$ then by the Rabinowitz Bifurcation Theorem \ref{thm:rabinowitz_bifurcation} there exists a sequence of numbers $\lambda_i \rightarrow 0$ and non-constant functions $U_i \rightarrow 0$ which solves equation \eqref{transformed_bifurcation_equation}.
\end{proof}

This ends the proof of Theorem \ref{thm:continous_ss_existence}.

\subsection{Proof of Theorem \ref{thm:discontinous_ss}}

We study the existence of a discontinuous solutions to problem \eqref{dispersion_equilibria}. In fact we start with a following Lemma.

\begin{lem}    
    Let $\Omega, J_\varepsilon$ satisfy either Assumption \ref{assumptions_general_kernel} or \ref{assumptions_periodic_kernel}. Assume that $f \in C^2_b(\mathbb{R})$ is such that there exists $u_1, u_2 \in \mathbb{R}$ satisfying $f(u_1) = f(u_2) = 0$ and $f'(u_1) \neq 0$, $f'(u_2) \neq 0$. Then there exists unique perturbation $\varphi$ such that $U = \Bar{U} + \varphi$ is a solution to steady problem \eqref{dispersion_equilibria} with $\hat{U}$ defined as in Theorem \ref{thm:discontinous_ss}.
    \label{lem:discontinuos_ss_existence}
\end{lem}

\begin{proof}
    Plug $U = \Tilde{U} + \varphi$ into equation \eqref{dispersion_equilibria} and re-write the obtained formula to get fixed point problem
    $$\mathcal{F}_\varepsilon \varphi = \frac{1}{f'(\Tilde{U})}\left(\Leo [\Tilde{U} + \varphi] + f(\Tilde{U} + \varphi) - f(\Tilde{U}) - f'(\Tilde{U})\varphi\right) = \varphi.$$
    We show that the mapping $\mathcal{F}_\varepsilon$ is a Banach contraction for some large value of $\varepsilon > 0$. For the Banach space we choose the closed ball $\overline{B(0,R)} \subset L^\infty(\Omega)$ with a small radius $R > 0$.
    Let $\varphi \in \overline{B(0,R)}$, then the norm is bounded as follows
    \begin{align*}
        \|\mathcal{F}_\varepsilon \varphi \|_\infty
        &\leq \left\|\frac{1}{f'(\Tilde{U})}\right\|_\infty\left( C(\Leo)(\|\Tilde{U}\|_\infty  + R) +  \|f''\|_\infty R^2\right).
    \end{align*}
    Here, the quadratic term is estimated using the Mean Value Theorem and the constant $C(\Leo)$ is given by the formula
    $$C(\Leo) := 2 \sup_{x \in \Omega} \int_\Omega J_\varepsilon(x-y) \:dy = O(\frac{1}{\varepsilon^m}).$$ 
    Thus in order to have $\mathcal{F}_\varepsilon: \overline{B(0,R)} \rightarrow \overline{B(0,R)}$ the following must be fulfilled
    \begin{equation}
        \left\|\frac{1}{f'(\Tilde{U})}\right\|_\infty\left( \frac{1}{\varepsilon^m}(\|\Tilde{U}\|_\infty  + R) +  \|f''\|_\infty R^2\right) \leq R. \label{condition_1}
    \end{equation}
    Now we derive the contractivity criterion. Take $\varphi, \psi \in \overline{B(0,R)}$ and estimate
    \begin{align*}
        \|\mathcal{F}_\varepsilon \varphi - \mathcal{F}_\varepsilon \psi\|_\infty &\leq \left\|\frac{1}{f'(\Tilde{U})}\right\|_\infty \left( C(\Leo) + \|f'(\zeta) - f'(\Tilde{U})\|_\infty \right) \|\varphi - \psi\|_\infty, \\
        &\leq \left\| \frac{1}{f'(\Tilde{U})}\right\|_\infty\left( \frac{1}{\varepsilon^m} + 2\|f'\|_\infty R \right)\|\varphi - \psi\|_\infty.
    \end{align*}
    The criterion of the following form has to hold true to get contractivity
    \begin{equation}
        \left\| \frac{1}{f'(\Tilde{U})}\right\|_\infty\left( \frac{1}{\varepsilon^m} + 2\|f'\|_\infty R \right) < 1.
        \label{condition_2}
    \end{equation}
    Observe that conditions \eqref{condition_1} and \eqref{condition_2} are satisfied for a sufficiently large values of dispersion parameter $\varepsilon > 0$ and small radius $R > 0$. Now the proof proceeds by the standard Banach fixed point argument.
\end{proof}

Next, we emphasize that the obtained discontinuous solution is stable.

\begin{lem}    
    Let $\Omega, J_\varepsilon$ satisfy either Assumption \ref{assumptions_general_kernel} or \ref{assumptions_periodic_kernel}. Assume that $f \in C^1(\mathbb{R})$ is such that there exists $u_1, u_2 \in \mathbb{R}$ satisfying $f(u_1) = f(u_2) = 0$ and $f'(u_1) < 0$, $f'(u_2) < 0$. Then solution $U$ given by Lemma \ref{lem:discontinuos_ss_existence} is linearly, exponentially stable in $L^2(\Omega)$.
    \label{lem:discontinuos_ss_stability}
\end{lem}

\begin{proof}
    By Lemma \ref{lem:discontinuos_ss_existence} the function $U = \Tilde{U} + \varphi$ is a steady state, where the jump function $\hat{U}$ and small perturbation $\psi$ are defined in the lemma. We linearise problem \eqref{disersion_evolution}-\eqref{disersion_initial_condition} around $U$ by introducing the new function $\psi = u - U$. This functions satisfies the following problem
    \begin{align*}
        \begin{cases}
            \psi_t(x,t) &= \Leo \psi(x,t) + f'(U(x))\psi(x,t), \\
            \psi(x,0) &= \psi_0(x),
        \end{cases}
    \end{align*}
    In order to proceed we define the energy of a linearized system as follows
    \begin{equation*}
        E_\phi(t) := \frac{1}{2} \int_\Omega |\phi(x,t)|^2 \:dx.
    \end{equation*}
    We verify that the eigenvalue of linearization operator is strictly negative
    \begin{align}
        \begin{split}
            \gamma_0 &= \sup_{\varphi \in \Leb, \varphi \neq 0} \frac{\langle \Leo \varphi + f'(U)\varphi, \varphi \rangle}{\|\varphi\|_2^2} \\
            &\leq \sup_{\varphi \in \Leb, \| \varphi \|_2 = 1} \langle f'(U)\varphi, \varphi \rangle \\
            &= \sup_{\varphi \in \Leb, \| \varphi \|_2 = 1} \int_\Omega f'(U) \varphi^2(x)\:dx < 0,
        \end{split}
        \label{gamma_0_negativity}
    \end{align}
    where last strict inequality holds true by continuity of $f'$ and perturbation $\varphi$ being defined in sufficiently small ball. Next we differentiate the energy $E_\varphi$.
    \begin{align*}
        E_\varphi'(t) &= \int_\Omega \varphi_t(x,t)\varphi(x,t)\:dx \\
        &= \int_\Omega \left(\Leo \varphi(x,t) + f'(U)\varphi(x,t)\right) \varphi(x,t)\:dx \\
        &= \frac{\langle \Leo \varphi + f'(U)\varphi, \varphi \rangle}{\|\varphi\|_2^2} \|\varphi\|_2^2 \\
        &\leq \gamma_0 E_\varphi(t).
    \end{align*}
    When solved, this inequality yields the following limit
    \begin{equation}
        \| \varphi(\cdot, t) - 0\|_2^2 =  E_\varphi(t) \leq \|\psi_0\|_2^2  e^{\gamma_0 t}.
        \label{energy_inequality}
    \end{equation}
    Inequality \eqref{gamma_0_negativity} and \eqref{energy_inequality} yields the stability of zero in a linearized problem.
\end{proof}
Combining lemmas \ref{lem:discontinuos_ss_existence} and \ref{lem:discontinuos_ss_stability} completes the proof of Theorem \ref{thm:discontinous_ss}.


\end{document}